\documentclass{amsart}
\usepackage{amssymb}
\usepackage{graphicx} 
\usepackage[x11names]{xcolor}
\usepackage{pinlabel} 
\usepackage{colortbl} 
\usepackage{enumitem} 
\usepackage{comment} 
\usepackage{appendix}
\usepackage[normalem]{ulem} 

\definecolor{darkblue}{rgb}{0,0,0.4} 
\usepackage{xr-hyper}
\usepackage[colorlinks=true, citecolor=darkblue, filecolor=darkblue, linkcolor=darkblue,urlcolor=darkblue]{hyperref} 
\usepackage[all]{hypcap}
\usepackage{crossreftools}

\newtheorem{theorem}{Theorem}[section]
\newtheorem{lemma}[theorem]{Lemma}
\newtheorem{proposition}[theorem]{Proposition}

\theoremstyle{remark}

\theoremstyle{definition}

\numberwithin{equation}{section}

\DeclareMathOperator{\lk}{lk}
\DeclareMathOperator{\ks}{ks}
\DeclareMathOperator{\Arf}{\mathrm{Arf}}
\DeclareMathOperator{\Int}{\mathrm{Int}}
\newcommand{\del}{\partial}

\newcommand{\Z}{\mathbb{Z}}

\renewcommand{\a}{\alpha}
\renewcommand{\b}{\beta}

\newcommand{\CP}[1]{\mathbb{CP}^{#1}}
\newcommand{\bCP}[1]{\overline{\mathbb{CP}^{#1}}}

\title{A note on smoothly slice links in $S^2 \times S^2$}
\author{Marco Marengon and Clayton McDonald}
\date{}

\begin{document}

\begin{abstract}
We give an alternative proof of a result of Miyazaki and Yasuhara that there exists links that are not smoothly slice in $S^2 \times S^2$. 
We discuss potential applications to the detection of exotic $S^2 \times S^2$.
This is a follow-up note to a similar paper for the $\CP2 \# \bCP2$ case.
\end{abstract}

\maketitle

\tableofcontents

\section{Introduction}

In this note, which is meant as a companion to \cite{MM:CP2CP2bar}, we give a proof in the smooth category that there exists a non-slice 2-component link in $S^2 \times S^2$.

\begin{theorem}
    \label{thm:S2xS2}
    The 2-component link in Figure \ref{fig:S2xS2link} is not smoothly slice in $S^2 \times S^2$.
\end{theorem}

\begin{figure}
    \centering    
    \includegraphics{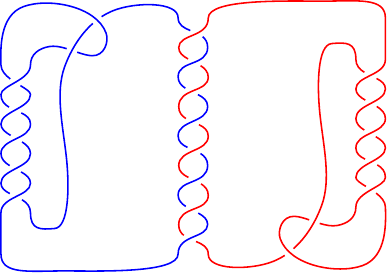}
    \caption{A 2-component link which is not smoothly slice in $S^2 \times S^2$.}
    \label{fig:S2xS2link}
\end{figure}

Recall that a link $L \subset S^3$ is (smoothly or topologically) \emph{slice in $X$} if it bounds (smoothly or locally flatly) embedded disjoint discs in $X \setminus B^4$.
A classical result \cite{N:slice, S:slice} shows that every knot is smoothly slice in both $S^2 \times S^2$ and $\CP2 \# \bCP2$.
In \cite{MM:CP2CP2bar} we produced a 2-component link that is not smoothly slice in $\CP2\#\bCP2$, and show that it can potentially lead to detecting an exotic $\CP2\#\bCP2$ (see \cite[Proposition 6.1]{MM:CP2CP2bar} for a precise statement).

Our method from \cite{MM:CP2CP2bar} can be adapted to produce links that are not slice in $S^2 \times S^2$, which in fact proves a result more general than our Theorem \ref{thm:S2xS2} (see Theorem \ref{thm:S2xS2assumptions}). We note that Miyazaki and Yasuhara have previously proved that there is a non-slice 2-component link in a punctured $S^2 \times S^2$ \emph{in the topological category} \cite{MY:generalized}.
While their result is technically stronger, it cannot be used to detect non-sliceness (in $S^2 \times S^2$) of our link in Figure \ref{fig:S2xS2link}, because their argument requires the signatures of the components to be $0 \pmod 8$.
More importantly, since their argument is topological, it cannot be used to detect exotic 4-manifolds.
By contrast, our proof uses the smooth genus function, so our link is not known to be non-slice in the topological category. This means we can potentially use it to construct exotic 4-manifolds, see Proposition \ref{prop:exotic}.

\subsection{Acknowledgements}
We thank Akira Yasuhara for pointing us towards earlier work on this topic.
We are very grateful to Andr\'as Stipsicz and Marco Golla for their support and helpful discussion.
We also thank the anonymous referees for their careful read and their suggestions, and M\'arton Beke for his comments.
The authors were partially supported by NKFIH grants K-146401 and Excellence-151337, and ERC Advanced Grant KnotSurf4d.

\section{Review of some obstructive methods}

As in \cite{MM:CP2CP2bar}, we review the statements of some obstructive methods that will be employed in our argument. We refer to \cite{MMP, MM:CP2CP2bar} for more details.

\subsection{Levine-Tristram signatures}

Here and below, we denote the signature function of a knot $K \subset S^3$ by $\sigma_K \colon S^1 \to \Z$.

\begin{theorem}[{\cite{V:G-signature, G:G-signature}, see \cite[Theorem 3.6]{MMP} for this statement}]
\label{thm:signature}
    Let $X$ be a topological closed oriented 4-manifold with $H_1(X;\Z)=0$.
    Let $\Sigma \subset X^\circ$ be a locally flat, properly embedded surface of genus $g$, with boundary a knot $K \subset S^3$.
    If the homology class $[\Sigma] \in H_2(X^\circ, \partial X^\circ;\Z) \cong H_2(X;\Z)$ is divisible by a prime power $m=p^k$, then
    \[
    \left| 
    \sigma_K(e^{2\pi r i/m}) + \sigma(X) - \frac{2r(m-r) \cdot [\Sigma]^2}{m^2}
    \right|
    \leq
    b_2(X) + 2g,
    \]
    for every $r=1, \ldots, m-1$.
\end{theorem}

We recall that the signatures of a satellite of a knot can be computed by the following formula.

\begin{theorem}[{\cite[Theorem 2]{L:signatures}}]
\label{thm:satellite}
    Let $C$ be a knot and $P$ be a pattern with winding number $w$. Then for every root of unity $\zeta$
    \[
    \sigma_{P(C)}(\zeta) = \sigma_{C}(\zeta^w) + \sigma_{P(U)}(\zeta),
    \]
    where $P(C)$ denotes the satellite of $C$ with pattern $P$, and $U$ denotes the unknot.
\end{theorem}

\subsection{Arf invariant}

In Theorem \ref{thm:Arf} below, $\ks(X)$ denotes the Kirby-Siebenmann invariant of a topological, closed 4-manifolds, and $\Arf(X, \Sigma)$ denotes the Arf invariant of $\Sigma$ in $X$ (see \cite{FK:Rochlin} for details). Recall that $\ks(X) \equiv 0$ if $X$ admits a smooth structure, and $\Arf(X, \Sigma) \equiv 0$ if $\Sigma$ is a disc. Also recall that a surface $\Sigma$ in a 4-manifold $X$ is characteristic if its homology class $[\Sigma] \in H_2(X; \Z)$ is as well, i.e., if for every $x \in H_2(X; \Z)$, $Q_X([\Sigma], x) \equiv Q_X(x,x) \pmod 2$, where $Q_X$ is the intersection form.

\begin{theorem}[{\cite{K:topology, Y:CL}, see \cite[Theorem 3.1]{MMP} for this statement}]
\label{thm:Arf}
Let $X$ be a topological, closed, connected, oriented 4-manifold.
If $\Sigma \subset X^\circ$ is a properly embedded, locally flat characteristic surface with boundary a knot $K$, then
\[
\frac{\sigma(X) - [\Sigma]^2}8 \equiv \Arf(K) + \Arf(X, \Sigma) + \ks(X) \pmod2
\]
\end{theorem}

\section{A 2-component link not smoothly slice in \texorpdfstring{$S^2 \times S^2$}{S2 x S2}}
\label{sec:S2xS2}

The proof of Theorem \ref{thm:S2xS2} is a rather long case analysis and will take the rest of the section. The structure follows closely \cite[Section 5]{MM:CP2CP2bar}.

For the rest of this section, let $X:=S^2 \times S^2$ and $X^\circ := X \setminus \Int(B^4)$.
Let $A$ and $B$ be the two link components, and suppose that $L$ bounds two disjoint smooth discs $D_A, D_B \subset X^\circ$, so that $\del D_A = A$ and $\del D_B = B$, and let $\a := [D_A]$ and $\b := [D_B]$ denote the homology classes of such discs in $H_2(X^\circ, S^3) \cong H_2(X)$.
The idea of the proof is to combine various obstructive methods to rule out all the possible pairs $(\a,\b)$, hence showing that the link cannot be slice.
To do so, we will progressively add assumptions on $L$ until we eventually eliminate all pairs $(\a,\b)$. All the assumptions are collected together in Appendix \ref{appendix:S2xS2}.

\begin{figure}
\begingroup%
  \makeatletter%
  \providecommand\color[2][]{%
    \errmessage{(Inkscape) Color is used for the text in Inkscape, but the package 'color.sty' is not loaded}%
    \renewcommand\color[2][]{}%
  }%
  \providecommand\transparent[1]{%
    \errmessage{(Inkscape) Transparency is used (non-zero) for the text in Inkscape, but the package 'transparent.sty' is not loaded}%
    \renewcommand\transparent[1]{}%
  }%
  \providecommand\rotatebox[2]{#2}%
  \newcommand*\fsize{\dimexpr\f@size pt\relax}%
  \newcommand*\lineheight[1]{\fontsize{\fsize}{#1\fsize}\selectfont}%
  \ifx\svgwidth\undefined%
    \setlength{\unitlength}{175.70183275bp}%
    \ifx\svgscale\undefined%
      \relax%
    \else%
      \setlength{\unitlength}{\unitlength * \real{\svgscale}}%
    \fi%
  \else%
    \setlength{\unitlength}{\svgwidth}%
  \fi%
  \global\let\svgwidth\undefined%
  \global\let\svgscale\undefined%
  \makeatother%
  \begin{picture}(1,0.54625635)%
    \lineheight{1}%
    \setlength\tabcolsep{0pt}%
    \put(0,0){\includegraphics[width=\unitlength,page=1]{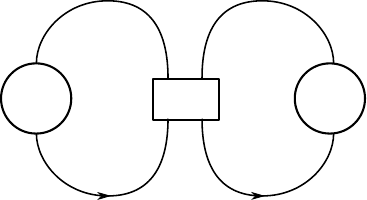}}%
    \put(0.07058651,0.26130829){\color[rgb]{0,0,0}\makebox(0,0)[lt]{\lineheight{1.25}\smash{\begin{tabular}[t]{l}\small $T_A$\end{tabular}}}}%
    \put(0.88102587,0.26130834){\color[rgb]{0,0,0}\makebox(0,0)[lt]{\lineheight{1.25}\smash{\begin{tabular}[t]{l}\small $T_B$\end{tabular}}}}%
    \put(0.49196079,0.26130829){\color[rgb]{0,0,0}\makebox(0,0)[lt]{\lineheight{1.25}\smash{\begin{tabular}[t]{l}\small $n$\end{tabular}}}}%
  \end{picture}%
\endgroup%

    \caption{The structure of the link $L$.}
    \label{fig:S2xS2structure}
\end{figure}
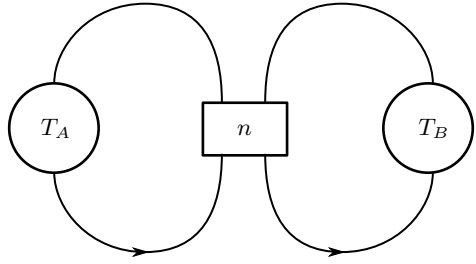

\subsection{The structure of the link \texorpdfstring{$L$}{L}}
We make some assumptions on the structure of $L$ to simplify our case analysis and our computations in the later subsections. Specifically, we assume that:
\begin{enumerate}[font=\textbf]
    \item[({\crtcrossreflabel{A1}[it:A1]})]
    $L$ has a diagram as in Figure \ref{fig:S2xS2structure}, where $T_A$ (resp.\ $T_B$) is a $(1,1)$-tangle whose closure is $A$ (resp.\ $B$), and $n \in \Z$ is the number of right-handed full twists added in the region.
\end{enumerate}
If $\Sigma_A, \Sigma_B \subset X^\circ$ are properly embedded surfaces in homology classes $\a$ and $\b$, and with boundary $A$ and $B$ respectively, then the following relation holds
\begin{equation}
\label{eq:lkformula}
    \#(\Sigma_A \pitchfork \Sigma_B) + \lk(m(L)) = \a \cdot \b,
\end{equation}
where $m(L)$ denotes the mirror of $L$. In particular, if $\Sigma_A$ and $\Sigma_B$ are disjoint, then \eqref{it:A1} implies that $n = - \lk(L) = \a \cdot \b$.

We recall the following lemma. In its statement, $K^r$ denotes the reverse of the knot $K$ (i.e.\ $K$ with reversed orientation), not to be confused with the mirror $m(K)$ of $K$, and the knot $K_{(p,q)}$ denotes the $(p,q)$-cable of $K$.

\begin{lemma}[{\cite[Lemma 5.1]{MM:CP2CP2bar}}]
\label{lem:sumsandcablings}
Suppose that $X$ is a smooth, connected 4-manifold, $L = A \cup B$ is a link satisfying \eqref{it:A1}, and that there are two disjoint smooth discs $D_A, D_B \subset X^\circ$, with $\del D_A = A$ and $\del D_B = B$. If $\a := [D_A]$ and $\b := [D_B]$, then:
\begin{itemize}
    \item the knot $A \# B$ bounds a smooth disc in $X^\circ$ in homology class $\a + \b$;
    \item the knot $A \# B^r \# T_{2, 2n\pm1}$ bounds a smooth disc in $X^\circ$ in homology class $\a - \b$;
    \item the knot $A \# (B_{(2,-2\beta^2-2n\pm1)})$ bounds a smooth disc in $X^\circ$ in homology class $\a + 2\b$.
\end{itemize}
Note that $n=-\lk(A,B)$.
\end{lemma}

\subsection{The genus function on \texorpdfstring{$S^2 \times S^2$}{S2 x S2}}

Given a smooth, connected 4-manifold $X$, the 4-ball genus of a knot $K$ gives a first obstruction to the homology classes of $H_2(X^\circ, S^3)$ that are represented by a disc with boundary $K$. 
More precisely, if there is such a disc in homology class $\alpha$, by gluing it to a minimal surface for $K$ in $B^4$ we obtain a closed surface of genus $g_{B^4}(K)$ sitting in a homology class that by abuse of notation we still call $\alpha \in H_2(X) \cong H_2(X^\circ, S^3)$.
Then, knowledge of the genus function on $H_2(X)$ can give obstruction to such an $\alpha$.


\begin{theorem}[{\cite[Corollary 1.3]{R:minimalgenus}}]
\label{thm:Ruberman}
    The minimal genus of a smoothly embedded orientable surface in $S^2 \times S^2$ in homology class $(a_1,a_2) \in H_2(S^2 \times S^2) \cong \Z^2$, with respect to the obvious basis, is
    \[
    G_{S^2 \times S^2}(a_1, a_2) =
    \begin{cases}
        0 & \mbox{if $a_1\cdot a_2=0$} \\
        (|a_1|-1) \cdot (|a_2|-1) & \mbox{otherwise}
    \end{cases}
    \]
\end{theorem}

Let us now return to the link $L = A \cup B$ whose sliceness in $S^2 \times S^2$ we are trying to obstruct. If either $A$ or $B$ were slice in the 4-ball, then a Kirby calculus argument shows that the link would be slice in $S^2 \times S^2$, see \cite[Proposition 1.3]{MM:CP2CP2bar}. The next simplest assumption we can make is the following:
\begin{enumerate}[font=\textbf]
    \item[({\crtcrossreflabel{A2}[it:A2]})]
    $g_{B^4}(A) = g_{B^4}(B) = 1$.
\end{enumerate}
This assumption implies that if $A$ bounds a smooth disc $D_A \subseteq X^\circ$, then $D_A$ must be in homology class $\a=(a_1,a_2)$ with $\min(|a_1|,|a_2|)\leq1$ or $|a_1|=|a_2|=2$. (Same goes for $D_B$.)

\subsection{Symmetries}

We will use symmetries to reduce the number of pairs of homology classes $(\a,\b)$ that we need to study. In addition to using the symmetries of $X:=S^2 \times S^2$, we will make the following assumption:
\begin{enumerate}[font=\textbf]
    \item[({\crtcrossreflabel{A3}[it:A3]})] The link $L$ has an ambient isotopy that swaps $A$ and $B$.
\end{enumerate}

With the above assumption, we list all the orientation-preserving symmetries we can use and their action:
\begin{enumerate}[font=\textbf]
    \item[({\crtcrossreflabel{S1}[it:S1]})]
    swapping the $S^2$ factors in $S^2 \times S^2$:
    \newline acts on $H_2(X)$ by $(a_1,a_2)\mapsto(a_2,a_1)$;
    \item[({\crtcrossreflabel{S2}[it:S2]})]
    inverting both $S^2$ factors:
    \newline acts on $H_2(X)$ by $(a_1,a_2)\mapsto(-a_1,-a_2)$;
    \item[({\crtcrossreflabel{S3}[it:S3]})]
    assumption \eqref{it:A3}:
    \newline acts on pairs $(\a,\b) \in H_2(X) \times H_2(X)$ by $(\a,\b) \mapsto (\b,\a)$.
\end{enumerate}

\begin{lemma}
\label{lem:A1-3}
Suppose that $L=A\cup B$ is a link satisfying \eqref{it:A1}-\eqref{it:A3} which is smoothly slice in $S^2 \times S^2$, with discs in homology classes $\a$ and $\b$. Then, up to symmetries we may assume that $\a$ (resp.\ $\b$) is as in the first column (resp.\ row) of Table \ref{tab:lk}, and that the corresponding table entry is not highlighted.
\end{lemma}

\begin{table}[]
    \centering
    \begin{tabular}{c|c|c|c|c|c}
     & $(0,y)$ & $(y,0)$ & $(\pm1,y)$ & $(y,\pm1)$ & $(\pm2,\pm2)$\\
    \hline
    $(0,x)$ & 0 & $xy$ & \cellcolor{Tan1}$\pm x$ & \cellcolor{Tan1}$xy$ & \cellcolor{Tan1}$\pm2x$ \\
    \hline
    $(1,x)$ & $y$ & $xy$ & $y \pm x$ & $xy \pm 1$ & \cellcolor{Tan1}$\pm2\pm2x$ \\
    \hline
    $(2,\pm2)$ & $2y$ & \cellcolor{Tan1}$\pm2y$ & $2y\pm2$ & \cellcolor{Tan1}$\pm2\pm2y$ & $0,\pm8$ \\
    \vspace{0pt}
    \end{tabular}
    \caption{The table shows the possible pairs of homology classes $(\a,\b)$, for $x,y \in \Z$, under assumptions \eqref{it:A1}-\eqref{it:A3}. The value of the cell is the intersection numbers $\a\cdot\b$. The highlighted cells (1,3),(1,4),(1,5), (2,5), (3,2) and (3,4) can be discarded by symmetry considerations.}
    \label{tab:lk}
\end{table}

\begin{proof}
By \eqref{it:A2}, we know that $g_{B^4}(A) = g_{B^4}(B) = 1$. We can cap off the slice discs $D_A$ and $D_B$ in $X^\circ$ with a minimum genus surface in $B^4$ for (the mirrors of) $A$ and $B$, and obtain closed surfaces of genus 1, smoothly embedded in $S^2 \times S^2$ in homology classes $\a$ and $\b$ respectively.

By Ruberman's result (Theorem \ref{thm:Ruberman}), we know that $\a$ and $\b$ must be of the form $(a_1,a_2)$ so that one of the following three conditions hold:
\begin{itemize}
    \item $|a_1| \leq 1$, or
    \item $|a_2| \leq 1$, or
    \item $|a_1|=|a_2|=2$.
\end{itemize}
This already gives the possible values of $\b$ in Table \ref{tab:lk}. For the possible values of $\a$, using symmetries \eqref{it:S1}-\eqref{it:S2}, we can further assume that $\a=(a_1,a_2)$ with $a_1=0,1,2$, and we therefore obtain Table \ref{tab:lk}.

The following cells are equivalent to other cells in the table via symmetries:
\begin{itemize}
\item cell (1,3) is equivalent to cell (2,1) via \eqref{it:S3} and possibly \eqref{it:S2};
\item cell (1,4) is equivalent to cell (2,2) via \eqref{it:S3}, \eqref{it:S1}, and possibly \eqref{it:S2};
\item cell (1,5) is equivalent to cell (3,1) via \eqref{it:S3} and possibly \eqref{it:S2};
\item cell (2,5) is equivalent to cell (3,3) via \eqref{it:S3} and possibly \eqref{it:S2};
\item cell (3,2) is equivalent to cell (3,1) via \eqref{it:S1} and possibly \eqref{it:S2};
\item cell (3,4) is equivalent to cell (3,3) via \eqref{it:S1} and possibly \eqref{it:S2}. \qedhere
\end{itemize}
\end{proof}

\subsection{Linking number and Arf invariant}

The next restriction that we will put on $L$ is the linking number of the two components, which is (up to sign) the intersection number $\a\cdot\b$. At the same time we also impose an assumption on the Arf invariant of $A$ and $B$, which will simplify our analysis.

\begin{enumerate}[font=\textbf]
    \item[({\crtcrossreflabel{A4}[it:A4]})]
    $\lk(A,B) = -4$.
    \item[({\crtcrossreflabel{A5}[it:A5]})]
    $\Arf A = \Arf B = 1$.
\end{enumerate}

\begin{lemma}
\label{lem:A1-5}
Suppose that $L=A\cup B$ is a link satisfying \eqref{it:A1}-\eqref{it:A5} which is smoothly slice in $S^2 \times S^2$, with discs in homology classes $\a$ and $\b$. Then, up to symmetries, the pair $(\a,\b)$ belongs to one of two infinite families
\begin{enumerate}
    \item \label{it:inf1} $((1,x), (1,4-x))$, for $x \in \Z$,
    \item \label{it:inf2} $((1,x), (-1,4+x))$, for $x \in \Z$,
\end{enumerate}
or is one of the following four sporadic cases:
\begin{enumerate}[resume]
    \item \label{it:spo1} $((2,2), (1,1))$,
    \item \label{it:spo2} $((2,2), (-1,3))$,
    \item \label{it:spo3} $((2,-2), (1,3))$,
    \item \label{it:spo4} $((2,-2), (-1,1))$.
\end{enumerate}
\end{lemma}
\begin{proof}
By Lemma \ref{lem:A1-3}, we know that the possible combinations of $\a$ and $\b$ are the entries of Table \ref{tab:lk} that are not highlighted. By \eqref{it:A4} and Equation \eqref{eq:lkformula}, we know that the intersection number $\a\cdot\b$, computed in Table \ref{tab:lk}, must be 4.

We study the possible combinations that yield $\a\cdot\b=4$ by analysing such an equation for each non-colored cell of the table separately. We will start with entry $(2,3)$ of the table, which yields the two infinite families.
\begin{itemize}
    \item Entry $(2,3)$: $y\pm x=4$.
    \newline The equation has infinitely many solutions, namely $y=4 \mp x$. Thus, we get the two infinite families in the statement of the lemma.
    \item Entry $(1,1)$: $0 = 4$.
    \newline The equation has no solutions.
    \item Entry $(2,1)$: $y = 4$.
    \newline This implies that $\b=(0,4)$, which is a characteristic class with $\b^2=0$. By Theorem \ref{thm:Arf}, we deduce that $\Arf B = 0$, contradicting \eqref{it:A5}. Thus, we do not get any new pair $(\a,\b)$.
    \item Entry $(3,1)$: $2y = 4$.
    \newline This implies that $\b=(0,2)$, which is again a characteristic class with $\b^2=0$. Thus, we do not get any new pair $(\a,\b)$.
    \item Entry $(1,2)$: $xy = 4$.
    \newline Every solution of the equations has either $x$ even or $y$ even, and therefore either $\a$ or $\b$ (or both) is characteristic and squares to 0, again contradicting \eqref{it:A5} as in the previous two points. Thus, no solutions of $xy = 4$ yield a possible pair $(\a,\b)$.
    \item Entry $(2,2)$: $xy = 4$.
    \newline Following the same reasoning of the previous points, we can rule out all solutions where $y$ is even. Thus, we get two possible solutions, namely $x=4, y=1$ and $x=-4, y=-1$. These yield the pairs $((1,4),(1,0))$ and $((1,-4),(-1,0))$. However, these are not new solutions, because they belong to the infinite families (1) and (2) respectively.
    \item Entry $(3,3)$: $2y\pm2 = 4$.
    \newline Depending on the choice of the signs in $(2,\pm2)$ and $(\pm1, y)$, this equation yields exactly the four sporadic cases as possible solutions.
    \item Entry $(2,4)$: $xy\pm1 = 4$.
    \newline The two equations are $xy=3$ and $xy=5$. Since both $3$ and $5$ are prime numbers, one of $x$ and $y$ must be $\pm1$. After perhaps applying symmetries \eqref{it:S1}-\eqref{it:S3}, we may assume that $y=\pm1$. Thus, we can assume that $\b$ is of the form $(\pm1, y')$, so any solution coming from this entry is already contained in the infinite families coming from entry $(2,3)$.
    \item Entry $(3,5)$: $0, \pm 8 = 4$.
    \newline This equation has no solutions.\qedhere
\end{itemize}
\end{proof}

\subsection{The genus function again}

From the list of cases provided by Lemma \ref{lem:A1-5} we can immediately rule out cases \eqref{it:inf1} and \eqref{it:spo1} using the genus function again.

\begin{lemma}
\label{lem:A1-5v2}
Under the assumptions of Lemma \ref{lem:A1-5}, the pair $(\a,\b)$ does not belong to the infinite family \eqref{it:inf1} or to the sporadic case \eqref{it:spo1}.
\end{lemma}

\begin{proof}
We claim that if $L = A \cup B$ were smoothly slice in $X$ with discs in homology classes $\a$ and $\b$, then the homology class $\a + \b$ would be represented by a closed, smoothly embedded, genus-2 surface.

Indeed, by Lemma \ref{lem:sumsandcablings} the knot $A \# B$ would bound a smooth disc in $X^\circ$ in homology class $\a + \b$. Since $g_4(A \# B) \leq 2$ by assumption \eqref{it:A2}, we would be able to construct the desired closed genus-2 surface, smoothly embedded in $X$ in homology class $\a + \b$, by capping off a $g_4$-minimising surface for $A\# B$ in $B^4$ with the slice disc in $X^\circ$.

However, if the pair $(\a,\b)$ belongs to the infinite family \eqref{it:inf1}, then $\a+\b = (2,4)$, and if $(\a,\b)$ belongs to the sporadic case \eqref{it:spo1}, then $\a+\b = (3,3)$. Neither of these homology classes are represented by a smooth genus-2 surface (see Theorem \ref{thm:Ruberman}).
\end{proof}

\subsection{Levine-Tristram signatures}
The last step in order to prove Theorem \ref{thm:S2xS2} is to obstruct the existence of pairs of discs $D_A$ and $D_B$ in the homology classes given by Lemma \ref{lem:A1-5}. We have already taken care of two cases in Lemma \ref{lem:A1-5v2}. For the remaining cases we will use the obstruction from the Levine-Tristram signatures (Theorem \ref{thm:signature}). We will choose a value of the classical signature of $A$ and $B$ to obstruct the infinite family \eqref{it:inf2} from Lemma \ref{lem:A1-5}, then we will use other signatures to obstruct the remaining sporadic cases.

We recall that $\sigma_K(\cdot)$ denotes the signature function of a knot $K$. We also introduce the notation $\zeta_m:=e^{2\pi i/m}$. We remark that $\zeta_m$ denotes this very specific root of unity, not just any primitive $m$-th root of unity. For example, $\sigma_K(\zeta_2) = \sigma_K(-1)$ is the classical signature of $K$.

These are the assumptions we make on the signature functions:

\begin{enumerate}[font=\textbf]
    \item[({\crtcrossreflabel{A6}[it:A6]})] $\sigma_A(\zeta_2) = \sigma_B(\zeta_2) = 2$.
    \item[({\crtcrossreflabel{A7}[it:A7]})] $\sigma_A(\zeta_4) = \sigma_B(\zeta_4) = 2$.
    \item[({\crtcrossreflabel{A8}[it:A8]})] $\sigma_A(\zeta_8) = \sigma_B(\zeta_8) = 2$.
\end{enumerate}

\begin{lemma}
    \label{lem:A6}
Suppose that $L=A\cup B$ is a link satisfying \eqref{it:A1} and \eqref{it:A6} which is smoothly slice in $S^2 \times S^2$, with discs in homology classes $\a$ and $\b$. Then the pair $(\a,\b)$ does not belong to the infinite family \eqref{it:inf2} of Lemma \ref{lem:A1-5}, nor is one of the sporadic cases \eqref{it:spo3} and \eqref{it:spo4}.
\end{lemma}

\begin{proof}
The obstruction that we use in this lemma is Theorem \ref{thm:signature}, which, in the special case of $K$ bounding a disc $D$ in $X^\circ$ such that $[D]$ is 2-divisible, says that
\begin{equation}
\label{eq:signatureS2xS2}
\left|\sigma_{K}(\zeta_2) -\frac{[D]^2}2 \right| \leq 2.
\end{equation}


Assume by contradiction that $(\a,\b)$ is in the infinite family \eqref{it:inf2}, i.e.\ it is of the form $((1,x),(-1,4+x))$ for some $x \in \Z$. By Lemma \ref{lem:sumsandcablings} the knot $A \# B$ bounds a smooth disc $D_+$ in $X^\circ$ in homology class $\a+\b=(0,4+2x)$.
By applying Equation \eqref{eq:signatureS2xS2} with $K=A\#B$ and $D=D_+$ we obtain
\[
\left|\sigma_{A\#B}(\zeta_2) -\frac{[D_+]^2}2 \right| \leq 2,
\]
which is a contradiction since $\sigma_{A\#B}(\zeta_2) = \sigma_{A}(\zeta_2) + \sigma_{B}(\zeta_2) = 4$ and $[D_+]^2 = 0$.

As for the sporadic cases \eqref{it:spo3} and \eqref{it:spo4}, we obstruct them by showing that $A$ cannot bound a disc $D_A$ in homology class $\a = (2,-2)$. Indeed, if such a disc existed, we could apply Equation \eqref{eq:signatureS2xS2}, which in this case yields the contradiction
\[
\left|2 -\frac{-8}2 \right| \leq 2. \qedhere
\]
\end{proof}

%

\begin{lemma}
    \label{lem:A7-8}
Suppose that $L=A\cup B$ is a link satisfying \eqref{it:A1}, \eqref{it:A4}, \eqref{it:A7} and \eqref{it:A8} which is smoothly slice in $S^2 \times S^2$, with discs in homology classes $\a$ and $\b$. Then $(\a,\b) \neq ((2,2),(-1,3))$, which is the sporadic case \eqref{it:spo2} of Lemma \ref{lem:A1-5}.
\end{lemma}

\begin{proof}
Assume by contradiction that $(\a,\b) = ((2,2),(-1,3))$.
By Lemma \ref{lem:sumsandcablings}, third bullet point, the knot $A\#(B_{(2,3)})$ bounds a smooth disc $D$ in $X^\circ$ in homology class $(0,8)$, which is 8-divisible. (For the computation of the cabling coefficient, we used $\b^2=-6$ and $n=4$.) Thus, we can apply Theorem \ref{thm:signature} with $m=8$ and $r=1$, and by noting that $[D]^2 = 0$ we obtain 
\begin{equation}
    \label{eq:lemA7-8}
\left|\sigma_{A\#(B_{(2,3)})} (\zeta_8)\right| \leq 2.
\end{equation}
We compute $\sigma_{A\#(B_{(2,3)})} (\zeta_8)$ using Theorem \ref{thm:satellite}:
\begin{align*}
\sigma_{A\#(B_{(2,3)})} (\zeta_8)
&= \sigma_{A} (\zeta_8) + \sigma_{B_{(2,3)}} (\zeta_8) \\
&= \sigma_{A} (\zeta_8) + \sigma_{B} (\zeta_4) + \sigma_{T_{2,3}} (\zeta_8) \\
&= 2 + 2 + 0 = 4,
\end{align*}
where in the last step we used \eqref{it:A7} and \eqref{it:A8}, and the computation for $T_{2,3}$ is straightforward from the definition.
Thus, we have obtained a contradiction with Equation \eqref{eq:lemA7-8}.
\end{proof}

\subsection{Proof of Theorem \ref{thm:S2xS2}}

Theorem \ref{thm:S2xS2} is a consequence of the following more general theorem.

\begin{theorem}
\label{thm:S2xS2assumptions}
Let $L$ be a 2-component link in $S^3$ satisfying assumptions \eqref{it:A1}-\eqref{it:A8}.
Then $L$ is not smoothly slice in $S^2 \times S^2$.
\end{theorem} 
\begin{proof}
Let $L = A \cup B$, and suppose by contradiction that it bounds two disjoint smooth discs in homology classes $\a$ and $\b$, respectively.
Then, by Lemma \ref{lem:A1-5} the pair $(\a,\b)$ belongs to one of the infinite families \eqref{it:inf1} and \eqref{it:inf2} or is one of the four sporadic cases \eqref{it:spo1}, \eqref{it:spo2}, \eqref{it:spo3}, and \eqref{it:spo4}.
However, Lemmas \ref{lem:A1-5v2}, \ref{lem:A6}, and \ref{lem:A7-8} show that none of these possibilities can happen. Thus, $L$ could not be smoothly slice in $S^2 \times S^2$.
\end{proof}

To find a concrete example of a 2-component link that is not slice in $S^2 \times S^2$, we need to produce a link satisfying all assumptions \eqref{it:A1}-\eqref{it:A8}.
In order to satisfy assumption \eqref{it:A3}, we will choose a knot $K$ with certain properties, and then set $A=B=K$. The properties that $K$ must satisfy are:
\begin{itemize}
    \item $g_4(K)=1$;
    \item $\Arf K = 1$;
    \item $\sigma_K(\zeta_2) = \sigma_K(\zeta_4) = \sigma_K(\zeta_8) = 2$.
\end{itemize}
A search with KnotInfo \cite{knotinfo} showed that we can choose $K=m(7_2)$, the mirror of the knot $7_2$.

Thus, since the link in Figure \ref{fig:S2xS2link} satisfies all assumptions \eqref{it:A1}-\eqref{it:A8}, Theorem \ref{thm:S2xS2assumptions} implies Theorem \ref{thm:S2xS2}.

\section{The search for exotic \texorpdfstring{$S^2 \times S^2$}{S2 x S2}'s}

Using a link $L = A \sqcup B$ which is not slice in $S^2 \times S^2$, produced from Theorem \ref{thm:S2xS2assumptions}, we can potentially give examples of an exotic $S^2 \times S^2$, using the proposition below.
In the next proposition $S^3_{f_A,f_B}(L)$ denotes the surgery on $L$ with framings $f_A$ and $f_B$, and $X_{f_A,f_B}(L)$ denotes the corresponding trace.

\begin{proposition}
\label{prop:exotic}
Suppose that $L = A \sqcup B$ is a link in $S^3$ which is not slice in $S^2 \times S^2$, and let $f_A$ and $f_B$ be framings for the two link components such that both $f_A$ and $f_B$ are even and the matrix
\[
Q :=
\left(
\begin{matrix}
    f_A & \lk \\
    \lk & f_B
\end{matrix}
\right)
\]
is of rank 2 and indefinite, where $\lk = \lk(A,B)$. 
If the rational homology sphere $Y = S^3_{f_A,f_B}(L)$ bounds a spin rational homology ball $W$ with $\pi_1(W)$ normally generated by $\pi(Y)$ such that $X := X_{f_A,f_B}(L) \cup -W$ is spin, then $X$ is an exotic $S^2 \times S^2$.
\end{proposition}

For example, every \emph{ribbon} rational homology ball (i.e.\ one that has a handle decomposition with no 3-handles) has $\pi_1$ normally generated by the boundary.

Unlike in $\CP2 \# \bCP2$ \cite{MM:CP2CP2bar}, the manifold $Y$ from Proposition \ref{prop:exotic} cannot be an integer homology sphere, as the surgery framings must both be even (because $S^2 \times S^2$ has an even intersection form) and the linking number is also even (by \eqref{it:A4}), so $\det Q$ is even, meaning that the order of $H_1(Y;\Z)$ is also even.

\begin{proof}

If $Y$ bounds an rational homology ball $W$ with $\pi_1(W)$ normally generated by $\pi(Y)$, then the result of the gluing $X:= X_{f_A, f_B}(L) \cup -W$ is simply-connected, by Van Kampen's theorem.
Thus, $X$ is homeomorphic to $S^2 \times S^2$ by Freedman's classification \cite{F:classification}. On the other hand, $X$ cannot be diffeomorphic to $S^2 \times S^2$, because $L$ is obviously smoothly slice in $X$ (by construction), but $L$ is not smoothly slice in $S^2 \times S^2$ by assumption. 
\end{proof}

\appendix

\section{Assumptions for the link not smoothly slice in \texorpdfstring{$S^2 \times S^2$}{S2 x S2}}
\label{appendix:S2xS2}
\begin{enumerate}[font=\textbf]
    \item[\eqref{it:A1}]
    $L$ has a diagram as in Figure \ref{fig:S2xS2structure}, where $T_A$ (resp.\ $T_B$) is a $(1,1)$-tangle whose closure is $A$ (resp.\ $B$), and $n \in \Z$ is the number of right-handed full twists added in the region.
    \item[\eqref{it:A2}] $g_{B^4}(A) = g_{B^4}(B) = 1$.
    \item[\eqref{it:A3}] The link $L$ has an ambient isotopy that swaps $A$ and $B$.
    \item[\eqref{it:A4}] $\lk(A,B) = -4$.
    \item[\eqref{it:A5}] $\Arf A = \Arf B = 1$.
    \item[\eqref{it:A6}] $\sigma_A(\zeta_2) = \sigma_B(\zeta_2) = 2$.
    \item[\eqref{it:A7}] $\sigma_A(\zeta_4) = \sigma_B(\zeta_4) = 2$.
    \item[\eqref{it:A8}] $\sigma_A(\zeta_8) = \sigma_B(\zeta_8) = 2$.
\end{enumerate}

\bibliographystyle{alpha}
\bibliography{bibliography}

\end{document}